\documentclass[12pt,a4paper,reqno]{amsart}
\usepackage{amsmath,amssymb,amsthm,mathtools}
\usepackage[utf8]{inputenc}
\usepackage[T1]{fontenc}
\usepackage{hyperref}
\usepackage{tikz-cd}
\usepackage[ margin=3.5cm ]{geometry}

\theoremstyle{plain}
\newtheorem{theorem}{Theorem}[section]
\newtheorem{lemma}[theorem]{Lemma}
\newtheorem{proposition}[theorem]{Proposition}
\newtheorem{corollary}[theorem]{Corollary}
\theoremstyle{definition}
\newtheorem{definition}[theorem]{Definition}
\newtheorem{remark}[theorem]{Remark}
\newtheorem{setup}[theorem]{Setup}

\theoremstyle{remark}
\newtheorem{example}[theorem]{Example}

\DeclareMathOperator{\modd}{mod}
\DeclareMathOperator{\Mod}{Mod}
\DeclareMathOperator{\Hom}{Hom}
\DeclareMathOperator{\End}{End}
\DeclareMathOperator{\add}{add}

\DeclareMathOperator{\op}{op}

\DeclareMathOperator{\pd}{pd}
\DeclareMathOperator{\Findim}{Findim}
\DeclareMathOperator{\findim}{findim}
\DeclareMathOperator{\dell}{dell}
\DeclareMathOperator{\subddell}{sub\text{-}ddell}
\DeclareMathOperator{\ddell}{ddell}
\DeclareMathOperator{\edell}{edell}
\DeclareMathOperator{\Dell}{Dell}
\DeclareMathOperator{\gddell}{gddell}

\newcommand{\D}{\mathcal{D}}
\newcommand{\Db}{\mathcal{D}^b}

\newcommand{\K}{\mathcal{K}}
\newcommand{\PR}{\mathcal{P}_R}
\newcommand{\PS}{\mathcal{P}_S}
\newcommand{\syz}{\Omega_{\mathcal{D}}}
\newcommand{\oemb}{\stackrel{\oplus}{\hookrightarrow}}

\title{Degree-shift derived invariance of derived delooping levels}

\author{Jiaqun Wei}
\address{School of Mathematical Sciences, Zhejiang Normal University, Jinhua 321004, Zhejiang, P. R. China}
\email{weijiaqun5479@zjnu.edu.cn}

\author{Kaili Wu}
\address{College of Science, Nanjing Forestry University, Nanjing 210037, Jiangsu, P. R. China}
\email{kailywu@163.com}

\author{Weiqing Cao}
\address{School of Mathematics and Statistics, Jiangsu Normal University, Xuzhou 221116, P.R. China}
\email{weiqingcao@jsnu.edu.cn}

\date{}

\begin{document}

\maketitle

\begin{abstract}
Gélinas introduced the delooping level of an Artin algebra as a homological invariant that bounds the big finitistic dimension of the opposite algebra. A natural question is whether such invariants are preserved under derived equivalences. Chen recently showed that the finiteness of the classical delooping level and its sub-derived variant is not derived-invariant, but the case of the finer derived delooping level of Guo and Igusa remained open. In this paper, we prove that the finiteness of the derived delooping level is \emph{degree-shift invariant} under derived equivalences: if two algebras are derived equivalent via a tilting complex of width $k_T$, then finiteness of the $(k+k_T)$-derived delooping level on one side implies finiteness of the $k$-derived delooping level on the other. Consequently, the finiteness of the derived \(\infty\)-delooping level is invariant under derived equivalences. Furthermore, we introduce the global derived delooping level and prove that, under a derived equivalence, its \(\infty\)-version changes by at most \(k_T\). In particular, its finiteness is a derived invariant.
\end{abstract}

\tableofcontents

\section{Introduction}

\subsection{The finitistic dimension and delooping levels}

The finitistic dimension conjectures, proposed by Bass \cite{Bass}, are central problems in the homological theory of Artin algebras. For an Artin algebra $\Lambda$, the \emph{big finitistic dimension} is
\[
\Findim(\Lambda) = \sup\{\pd(M) \mid M \in \Mod \Lambda, \pd(M) < \infty\},
\]
and the \emph{little finitistic dimension} is
\[
\findim(\Lambda) = \sup\{\pd(M) \mid M \in \modd \Lambda, \pd(M) < \infty\}.
\]
The finitistic dimension conjecture asserts that $\findim(\Lambda) < \infty$ for all Artin algebras $\Lambda$.

In an attempt to resolve this conjecture, Gélinas \cite{Gelinas} introduced the \emph{delooping level} $\dell(\Lambda)$, a syzygy-theoretic invariant that bounds the big finitistic dimension of the opposite algebra:
\begin{equation}\label{eq:gelinas-bound}
\Findim(\Lambda^{\op}) \leq \dell(\Lambda).
\end{equation}

Ringel \cite{RingelNakayama} proved that for Nakayama algebras, the finitistic 
dimension coincides with the delooping level. Subsequently, Sen \cite{Sen} gave 
another proof of this result using syzygy filtrations. Kershaw and Rickard \cite{KershawRickard} constructed the first example of a finite-dimensional algebra with infinite delooping level.

The delooping level was refined by Guo and Igusa \cite{GuoIgusa}, who introduced the $k$-\emph{delooping level} $k$-$\dell$, the \emph{effective delooping level} $\edell$, the \emph{sub-derived delooping level} $\subddell$, and the $k$-\emph{derived delooping level} $k$-$\ddell$. The $k$-derived delooping level of $\Lambda$ is $k$-$\ddell(\Lambda) = \sup\{k$-$\ddell(S) \mid S \text{ simple}\}$.
 Usually, one sets $\ddell = 1$-$\ddell$.
 
 These invariants give new bounds for the finitistic dimension and satisfy
\begin{equation}\label{eq:ddell-chain}
\Findim(\Lambda^{\op}) = \edell(\Lambda) \leq \ddell(\Lambda) \leq \dell(\Lambda).
\end{equation}

Subsequently, 
Guo \cite{GuoSerial} investigated the relationship between the delooping level and the finitistic dimension for serial path algebras, showing that for right serial algebras the right finitistic dimension equals the left delooping level.
The symmetry of the derived delooping level was investigated in \cite{SymDDell}.
Chen, Li, Zhang and Zhao \cite{ChenLiZhangZhao} studied the relation between $\tau$-tilting modules and 1-tilting modules, and used the delooping level to prove that a self-orthogonal $\tau$-tilting module is 1-tilting provided its endomorphism algebra has finite global delooping level.
Barrios, Lanzilotta and Mata \cite{BLM} introduced the global delooping level and studied its relation with the $\phi$-dimension.  Xu and Zhang \cite{XuZhang} generalized the delooping level to the setting of $\tau$-tilting modules, defining the depth and delooping level relative to a $\tau$-tilting module, and established bounds for the finitistic dimension of the endomorphism algebra.

\subsection{Derived invariance}

 Rickard \cite{Rickard} introduced tilting complexes, which characterize derived equivalences. Derived equivalences preserve many homological properties. Pan and Xi \cite{PanXi2009} proved that the finiteness 
of the finitistic dimension is preserved.  
Wei \cite{Wei} proved that syzygy-finite algebras, Igusa--Todorov algebras, and AC-algebras are preserved under derived equivalences. Fernandes, Lanzilotta and Hern\'andez \cite{MendozaSaenz} obtained that derived equivalences also preserve the finiteness of the $\phi$-dimension, among others.

It is therefore natural to ask whether delooping levels are preserved under derived equivalences. Xi and Zhang conjectured that the finiteness of $\dell$ is invariant under derived equivalences among Noetherian rings \cite{XiZhang}. However, Chen \cite{Chen} recently gave a counterexample with the help of the Kershaw--Rickard example \cite{KershawRickard}. Indeed, he constructs two derived equivalent finite-dimensional algebras $B$ and $C$ over a field, where $B$ is the Kershaw--Rickard algebra \cite{KershawRickard} with infinite delooping level, such that
\[
\dell(B) = \subddell(B) = \infty \quad \text{and} \quad \dell(C) = \subddell(C) = 0.
\]

\subsection{Our results}

The main question left open by Chen's counterexample is whether the finiteness of the derived delooping level is preserved under derived equivalences. 

To study the behavior of derived delooping levels, we first clarify the important role of the number $k$ by calling it the \emph{delooping degree}. With this notion, the statement that a module $M$ satisfies $k$-$\ddell(M)\le n$ means that $M$ admits a derived delooping level $n$ with degree $k$. 

Our main theorem can be formulated as follows.

\begin{theorem}\label{thm:main-intro}
The finiteness of derived delooping levels is degree-shift invariant under derived equivalence. More precisely, 
let $R$ and $S$ be derived equivalent Artin algebras, and fix an integer $k\ge 1$. If the derived equivalence is represented by a tilting complex of width $k_T$, then $(k+k_T)$-$\ddell (S)<\infty$ implies $k$-$\ddell (R)<\infty$.
\end{theorem}

Following \cite{GuoIgusa}, we say that a module is infinitely deloopable if all its delooping levels (for all degrees) are $0$. Similarly, $\infty$-$\ddell$ means $k$-$\ddell$ for all $k\ge 1$. We then have the following corollary, which shows that the finiteness of the $\infty$-derived delooping level is indeed derived invariant.

\begin{corollary} \label{cor:main-intro}
Let $R$ and $S$ be derived equivalent Artin algebras. Then $\infty$-$\ddell (S)<\infty$ if and only if $\infty$-$\ddell (R)<\infty$.
\end{corollary}

Define the \emph{global derived delooping level with degree
$k$} of an Artin algebra $R$ as
\[
k\text{-}\gddell(R)
:=\sup\{k\text{-}\ddell(M)\mid M\in\modd R\}.
\]
This is the global version of Guo--Igusa's derived delooping level.
A similar global version of Gélinas's delooping level, $\Dell(R):=\sup\{\dell(M)\mid M\in\modd R\},$ 
was introduced by Barrios--Lanzilotta--Mata \cite{BLM}.  Thus
$k\text{-}\gddell$ refines their $\Dell$ by replacing $\dell$ with
the sharper invariant $\ddell$ and keeping track of the delooping
degree $k$.

We prove that $k\text{-}\ddell(R)<\infty$ if and only if  $k\text{-}\gddell(R)<\infty$ (see  Lemma~\ref{lem:global-ddell} ). Combining this local-to-global equivalence with
Theorem~\ref{thm:main-intro}, we obtain the following results which is similar to the well-known case of global dimension.

\begin{corollary}\label{cor:gddell-intro}
let $R$ and $S$ be derived equivalent Artin algebras, and fix an integer $k\ge 1$. If the derived equivalence is represented by a tilting complex of width $k_T$, then 
\[
\left|
\infty\text{-}\gddell(R)
-
\infty\text{-}\gddell(S)
\right|
\leq k_T.
\]
\end{corollary}

\subsection{Structure of the paper}

The paper is organized as follows. In Section~\ref{sec:prelim}, we recall the necessary background on derived categories and syzygy complexes. In Section~\ref{sec:delooping}, we review classical, and derived delooping levels, recall the non-invariance of $\dell$ and $\subddell$ under derived equivalences, and introduce delooping degrees and derived delooping resolutions. Section~\ref{sec:invariance} contains the main results. Finally, Section~\ref{sec:examples} presents classes of algebras with finite derived $\infty$-delooping level and discusses exact values and applications under derived equivalences.

\section{Notation and preliminaries}\label{sec:prelim}

\subsection{Conventions and derived equivalences}

Throughout this paper, $R$ and $S$ denote Artin algebras over a commutative Artinian ring. Complexes are chain complexes. We always assume modules are finitely generated and right modules unless otherwise stated.

\begin{itemize}
\item $\modd R$: the category of (finitely generated right) $R$-modules.
\item $\PR$: the subcategory of projective modules in $\modd R$.
\item $\Db(R)$: the bounded derived category.
\item $\D_{[a,c]}(R)$: the full subcategory of $\Db(R)$ consisting of complexes with homology concentrated in degrees $a$ to $c$.
\item $\K^{-,b}(\PR)$: the homotopy category of right-bounded complexes of projectives with bounded homology.
\item $\K_{[a,b]}(\PR)$: complexes of projectives concentrated in degrees $a$ to $b$.
\item $[n]$: the shift functor, $(X[n])_i = X_{i-n}$ and $d_{X[n]} = (-1)^n d_X$.
\item $\sigma_I(M)$: the brutal truncation of $M \in \K^{-,b}(\PR)$ by setting $\sigma_I(M)_i = M_i$ for $i\in I$ and $0$ otherwise.
\item $\syz^n(M)$: the $n$-th syzygy complex of $M\in \Db(R)$ (Definition \ref{def:syzygy}).
\end{itemize}

Two algebras $R$ and $S$ are \emph{derived equivalent} if $\Db(R)$ and $\Db(S)$ are equivalent as triangulated categories. By Rickard's theorem \cite{Rickard}, this is equivalent to the existence of a tilting complex $T \in \K^b(\PR)$ with $\End_{\D(R)}(T) \cong S$ and $\add T$ generating $\K^b(\PR)$.

\subsection{Syzygy complexes}

We recall the definition and basic properties of syzygy complexes in derived categories, following Avramov--Iyengar \cite{AI} and Wei \cite{Wei}.

\begin{definition}\label{def:syzygy}
Let $M \in \Db(R)$ and $n \in \mathbb{Z}$. Let $P_M$ be a projective resolution of $M$. The \emph{$n$-th syzygy complex} of $M$ is
\[
\syz^n(M) := (\sigma_{\geq n}(P_M))[-n] \in \Db(R).
\]
\end{definition}

\begin{remark}\label{rem:syzygy-well-defined}
The syzygy complex depends a priori on the choice of projective resolution. However, by \cite[Proposition 3.5]{Wei}, any two choices yield projectively equivalent complexes, i.e., they are isomorphic up to direct summands of projective modules. \emph{When we say two syzygies are isomorphic, we mean this in the above sense}.
\end{remark}

\begin{remark}\label{rem:classical-syzygy}
When $M$ is an $R$-module (viewed as a complex concentrated in degree $0$) and $n \ge 0$, the $n$-th syzygy complex $\syz^n(M)$ coincides with the usual $n$-th syzygy module in $\modd R$. In particular, $\syz^0(M) \simeq M$.
\end{remark}

The following lemma collects the basic properties we need.

\begin{lemma}\label{lem:syzygy-basic}
Let $M, N \in \Db(R)$ and $n, m, a, b, j$ be integers. Then:
\begin{enumerate}
\item $\syz^n(M) \in \D_{[0,\infty)}(R) \cap \Db(R)$ and $\Hom_{\D(R)}(Q, \syz^n(M)[j]) = 0$ for any projective module $Q$ and $j > 0$.

\item If $M \in \D_{[a,b]}(R)$ for $a \leq b$, then:
\[
\syz^n(M) \in 
\begin{cases}
\D_{[0,0]}(R) = \modd R, & n \geq b, \\
\D_{[0,b-n]}(R), & a \leq n \leq b, \\
\D_{[a-n,b-n]}(R), & n \leq a.
\end{cases}
\]
In particular, $\syz^b(M) \in \modd R$ and $M \simeq \syz^a(M)[a]$.

\item $\syz^{n+m}(M[m]) \simeq \syz^n(M)$.

\item $\syz^{n+m}(M) \simeq \syz^m(\syz^n(M))$ for $m \geq 0$.

\item $\syz^n(M) \oplus Q$ is also an $n$-th syzygy of $M$ for any projective $Q$.

\item $M \in \K^b(\PR)$ iff any (equivalently, some) syzygy of $M$ lies in $\K^b(\PR)$.

\item $\syz^n(M \oplus N) \simeq \syz^n(M) \oplus \syz^n(N)$.
\end{enumerate}
\end{lemma}

\begin{proof}
These are standard; see \cite[Lemma 3.3]{Wei}. For completeness, we note that (2) follows from the definition of brutal truncation and the fact that $H_i(\sigma_{\geq n}P) = H_i(P)$ for $i \ge n$ and $0$ otherwise. Property (4) follows from the observation that $\sigma_{\geq n+m}(P)[-(n+m)] \simeq \sigma_{\geq m}(\sigma_{\geq n}P)[-m]$, which is a standard truncation identity.
\end{proof}

\begin{lemma}\label{lem:syzygy-triangle}
Let $M \in \Db(R)$ and $n \geq m$ be integers. Then there exists a triangle
\[
(\syz^{n+1}(M))[n-m] \longrightarrow B \longrightarrow \syz^m(M) \longrightarrow,
\]
where $B \in \K_{[0,n-m]}(\PR)$. In particular, for any integer $n$, there is a triangle
\[
\syz^{n+1}(M) \longrightarrow Q \longrightarrow \syz^n(M) \longrightarrow,
\]
with $Q$ a projective $R$-module.
\end{lemma}

\begin{proof}
See \cite[Lemma 3.4]{Wei}. The triangle is obtained by taking the short exact sequence of complexes
\[
0 \to \sigma_{\geq n+1}P \to \sigma_{\geq n}P \to P_n[-n] \to 0,
\]
where $P_n$ is the $n$-th term of the projective resolution $P$, and applying the appropriate shifts.
\end{proof}

\begin{lemma}\label{lem:syzygy-comparison}
Let $L \to B \to N \to$ be a triangle in $\Db(R)$. If $B \in \K_{(-\infty,t]}(\PR)$ for some integer $t$, then
\[
\syz^t(L) \simeq \syz^{t+1}(N).
\]
Consequently, $\syz^n(L) \simeq \syz^{n+1}(N)$ for all $n \geq t$.
\end{lemma}

\begin{proof}
See \cite[Lemma 3.7]{Wei}. The proof uses the fact that if $B$ is a bounded-above complex of projectives, then $\syz^n(B)$ is projective for all $n \ge k$.
\end{proof}

\begin{lemma}\label{lem:syzygy-preserve}
Let $L \to M \to N \to$ be a triangle in $\Db(R)$. Then for any integer $n$, there exists a triangle
\[
\syz^n(L) \longrightarrow \syz^n(M) \longrightarrow \syz^n(N) \longrightarrow.
\]
\end{lemma}

\begin{proof}
See \cite[Proposition 3.8]{Wei}. The proof uses the fact that $\syz^n$ is induced by a functor on the homotopy category which is exact in the triangulated sense.
\end{proof}

\section{Delooping levels with degree}\label{sec:delooping}

In this section we review the three main delooping-type invariants and discuss their behavior under derived equivalences.

\subsection{Classical delooping level}

\begin{definition}\label{def:classical-dell}
Let $M \in \modd R$. The \emph{delooping level} of $M$ is
\[
\dell(M) = \inf\{n \geq 0 \mid \Omega^n(M)  \oemb  \Omega^{n+1}(N) \text{ for some } N \in \modd R\}.
\]
If no such $n$ exists, we set $\dell(M) = \infty$. The \emph{delooping level} of $R$ is $\dell(R) = \sup\{\dell(S) \mid S \text{ simple}\}$.
\end{definition}

\subsection{Sub-derived and derived delooping levels}

Guo and Igusa \cite{GuoIgusa} introduced refinements that close the gap between $\dell$ and $\Findim$.

\begin{definition}\label{def:subddell}
The \emph{sub-derived delooping level} of $M \in \modd R$ is
\[
\subddell(M) = \inf\{\dell(N) \mid M \text{ is a direct summand of } N \text{ in } \modd R\}.
\]
The \emph{sub-derived delooping level} of $R$ is $\subddell(R) = \sup\{\subddell(S) \mid S \text{ simple}\}$.
\end{definition}

\begin{definition}\label{def:k-dell}
Let $M \in \modd R$. The \emph{$k$-delooping level} of $M$ is
\[
k\text{-}\dell(M) = \inf\{n \geq 0 \mid \Omega^n(M)  \oemb  \Omega^{n+k}(N) \text{ for some } N \in \modd R\}.
\]
If no such $n$ exists, we set $k\text{-}\dell(M) = \infty$. The \emph{$k$-delooping level} of $R$ is $k\text{-}\dell(R) = \sup\{k\text{-}\dell(S) \mid S \text{ simple}\}$.
\end{definition}

It is easy to see that $\dell(M)\le k\text{-}\dell(M)\le (k+1)\text{-}\dell(M)$ for $k\ge 1$.

\begin{definition}\label{def:ddell}
For $k\ge 1$, the \emph{$k$-derived delooping level} of $M \in \modd R$ is
\[
k\text{-}\ddell(M) = \inf \left\{ m \in \mathbb{N} \;\middle|\;
\begin{aligned}
&\exists n \leq m \text{ and an exact sequence in } \modd R \\
&\quad 0 \to C_n \to C_{n-1} \to \cdots \to C_1 \to C_0 \to M \to 0, \\
&\text{where } (i+k)\text{-}\dell(C_i) \leq m-i \text{ for } i=0,1,\ldots,n
\end{aligned}
\right\}.
\]
The \emph{$k$-derived delooping level} of $R$ is $k\text{-}\ddell(R) = \sup\{k\text{-}\ddell(S) \mid S \text{ simple}\}$.
\end{definition}

As in \cite{GuoIgusa}, we drop the $k$ when $k=1$ and write $\ddell(M)$ instead of $1\text{-}\ddell(M)$.
The following result collects the relations between the delooping invariants.

\begin{corollary}\label{cor:compare-k-dell} Let $k\ge 1$ and $M\in \modd R$.

   $(1)$ $k\text{-}\ddell(M)\le k\text{-}\dell(M)$ and $\ddell(M)\le k\text{-}\ddell(M) \le (k+1)\text{-}\ddell(M)$. 
   
   $(2)$ $k\text{-}\ddell(R)\le k\text{-}\dell(R)$ and  $\ddell(R)\le k\text{-}\ddell(R)\le (k+1)\text{-}\ddell(R)$. 
\end{corollary}

We refer to \cite[Section 3]{GuoIgusa} for further properties of these invariants. One of them is that $k\text{-}\ddell(R)<\infty$ is equivalent to the fact that $k\text{-}\ddell(M)<\infty$ for each $M\in \modd R$.

\subsection{The non-invariance of $\dell$ and $\subddell$ under derived equivalences}

Chen's counterexample \cite{Chen} shows that neither $\dell$ nor $\subddell$ has finiteness preserved under derived equivalences. We recall the construction.

\begin{example}\label{ex:kershaw-rickard}
Let $K$ be a field with $q \in K^\times$ of infinite multiplicative order. Let
\[
\Lambda = K\langle x,y,z\rangle/(x^2, y^2, z^2, yz, xy+qyx, xz-zx, zy-zx).
\]
For $\alpha \in K$, let $M(\alpha)$ be the 3-dimensional left $\Lambda$-module with basis $v,v',v''$ where $xv = \alpha v'$, $yv = v'$, $zv = v''$. Let $B = \begin{pmatrix} K & 0 \\ M(q) & \Lambda \end{pmatrix}$.

Then $\dell(B) = \subddell(B) = \infty$. However, $B$ is derived equivalent to $C = \begin{pmatrix} \Lambda & D(M(q)) \\ 0 & K \end{pmatrix}$, and $\dell(C) = \subddell(C) = 0$.
\end{example}

\subsection{Degree and derived delooping resolutions}\label{sec:derived-ddell}

We introduce the following notions.

\begin{definition}\label{def:classical-ddell}
$(1)$ We call the integer $k$ in Definition \ref{def:k-dell} the \emph{delooping degree}. Thus, $k\text{-}\dell(M)$ means the delooping level of $M$ with degree $k$.
 
$(2)$ Let $M \in \modd R$. We say that $M$ admits a \emph{derived delooping resolution $($of length $n$$)$ with degree $k$}, or shortly \emph{$k$-$\ddell$ resolution}, if there is a finite exact sequence in $\modd R$:
\[
0 \longrightarrow D_n \longrightarrow D_{n-1} \longrightarrow \cdots \longrightarrow D_1 \longrightarrow D_0 \longrightarrow M \longrightarrow 0
\]
such that for each $i \in \{0, 1, \ldots, n\}$, there exists a module $N_i \in \modd R$ with
\[
\Omega^{n-i}(D_i) \oemb \Omega^{n+k}(N_i),
\]
meaning that $\Omega^{n-i}(D_i)$ is a direct summand of $\Omega^{n+k}(N_i)$, i.e., $(i+k)$-$\dell(D_i)\le n-i$.
\end{definition}

We have the following characterization of derived delooping levels with degree $k$ in terms of derived delooping resolutions with degree $k$.

\begin{lemma}\label{lem:k-ddell}
For a fixed integer $k \ge 1$ and $M \in \modd R$, the following holds:
\[
k\text{-}\ddell(M) = \inf\{\, n \ge 0 \mid \text{\(M\) admits a \(k\)-\(\ddell\) resolution of length \(n\)} \,\}.
\]
Thus $k\text{-}\ddell(M)$ is the minimal length of a $k$-$\ddell$ resolution of $M$.
\end{lemma}

\begin{proof}
This is immediate from Definition \ref{def:ddell}: an admissible exact sequence of length $n$ gives $k\text{-}\ddell(M)\le n$, and the definition of $k\text{-}\ddell(M)$ as an infimum over $m$ is attained at the length of the shortest resolution. Hence the two infima coincide.
\end{proof}

The following result will be used in the next section.

\begin{lemma}\label{lem:syz-k-ddell} 
Let $M \in \modd R$, $k \ge 1$ and $n\ge 0$.

$(1)$ If $M$ admits a \(k\)-\(\ddell\) resolution of length \(m\), then $\Omega^{n}M$ admits an \((n+k)\)-\(\ddell\) resolution of length \(m\).

$(2)$ If $\Omega^{n}(M)$ admits a \(k\)-\(\ddell\) resolution of length \(m\) with $k>n$, then $M$ admits a \((k-n)\)-\(\ddell\) resolution of length \(m'\) such that \(m\le m'\le m+n\).
\end{lemma}

\begin{proof}
The case $n=0$ is obvious, since $\Omega^{0}M=M$. The case $n=1$ is \cite[Lemmas 3.2 and 3.3]{GuoIgusa}. The general case follows by induction on $n$.
\end{proof}

\section{Degree-shift derived invariance of finiteness of $\ddell$}\label{sec:invariance}

We now prove the main theorem. We work under the following setup throughout this section.

\begin{setup}\label{setup:derived}
Let \(\mathcal{F}: \Db(R) \to \Db(S)\) be a derived equivalence with quasi-inverse \(\mathcal{G}: \Db(S) \to \Db(R)\). Let \(T := \mathcal{G}(S) \in \K^b(\PR)\) be the tilting complex with width $k_T$. Thus, after a shift, we assume \(T \in \K_{[-k_T,0]}(\PR)\).
\end{setup}

\subsection{Image formulas}

The following lemmas control the homological support under derived equivalences.

\begin{lemma}\label{lem:FR} $
\mathcal F(R) \subseteq \K_{[0,k_T]}(\mathcal P_S).$
\end{lemma}
\begin{proof}
As $T:=\mathcal G(S)\in \mathcal K_{[-k_T,0]}(\mathcal P_R)$ is a tilting complex, its dual complex $\operatorname{Hom}_R(T,R)$ has terms in degrees $0,\dots,k_T$. Since $\mathcal F$ is the quasi-inverse of $\mathcal G$, the object $\mathcal F(R)$ is just the dual of $T$, hence $\mathcal F(R)\in \mathcal K_{[0,k_T]}(\mathcal P_S)$. 
\end{proof}

\begin{lemma}\label{lem:image-bounded}
Let $a \leq b$. Then 

$(1)$ $ \mathcal{F}(\D_{[a,b]}(R)) \subseteq \D_{[a,b+k_T]}(S)$.
In particular, $\mathcal{F}(\modd R) \subseteq \D_{[0,k_T]}(S)$.

$(2)$ \(
\mathcal G(\D_{[a,b]}(S)) \subseteq \D_{[a-k_T,\, b]}(R).
\)
In particular, \(\mathcal G(\mathrm{mod}\,S)\subseteq \D_{[-k_T,0]}(R)\).
\end{lemma}

\begin{proof}
(1) For $M \in \D_{[a,b]}(R)$,
\[
H_i(\mathcal{F}(M)) \cong \Hom_{\D(S)}(S, \mathcal{F}(M)[-i]) \cong \Hom_{\D(R)}(T, M[-i]).
\]
The isomorphism follows from the adjunction between $\mathcal{F}$ and $\mathcal{G}$. Since $T \in \K_{[-k_T,0]}(\PR)$ and $M \in \D_{[a,b]}(R)$, the right-hand side vanishes unless $i \in [a, b+k_T]$. 

(2) Dually (using Lemma \ref{lem:FR}). 
\end{proof}

\begin{lemma}\label{lem:image-projective}
Let $a \leq b$. Then 

$(1)$ $\mathcal{G}(\K_{[a,b]}(\PS)) \subseteq \K_{[a-k_T,b]}(\PR)$.

$(2)$ $\mathcal F(\K_{[a,b]}(\mathcal P_R)) \subseteq \K_{[a,\; b+k_T]}(\mathcal P_S).$ 
\end{lemma}

\begin{proof} (1)
Note that $\mathcal{G}(\K_{[a,b]}(\PS)) = \mathcal{G}(\K_{[0,b-a]}(\PS)[a])$.  For any $X\in \K_{[0,b-a]}(\PS)$, there are triangles $$\Omega^{i+1}(X)\to P_i\to \Omega^{i}(X) \to, \hskip 20pt i=0,\cdots,b-a-1,$$ with all $P_i$ projective,  $P_{b-a}=\Omega^{b-a}(X)$ projective and $X = \Omega^0(X)$. Applying $\mathcal{G}$ to these triangles and noting that $\mathcal{G}(P_i)\in \K_{[-k_T,0]}(\PR)$ and that the mapping cone moves the left side of the complex further one step while keeping the right side, we obtain that $\mathcal{G}(X)\in \K_{[-k_T,b-a]}(\PR)$. It follows that $\mathcal{G}(\K_{[a,b]}(\PS)) \subseteq \K_{[a-k_T,b]}(\PR)$. 

(2) Dually.
\end{proof}

\subsection{The syzygy descent lemma}

The following is the key technical tool for transferring syzygy properties across a derived equivalence.

\begin{lemma}\label{lem:syzygy-descent}
Let $i$ be an integer.

$(1)$ \(
\syz^{j+1}(\mathcal{G}(\syz^i(C))) \simeq \syz^{j}(\mathcal{G}(\syz^{i+1}(C))) \text{ holds for any } C \in \Db(S) \text{ and any } j\ge 0.
\)

$(2)$ \(
\syz^{j+1}(\mathcal{F}(\syz^i(C))) \simeq \syz^{j}(\mathcal{F}(\syz^{i+1}(C))) \text{ holds for any } C \in \Db(R) \text{ and any } j\ge k_T.
\)
\end{lemma}

\begin{proof} (1)
By Lemma \ref{lem:syzygy-triangle}, for $C\in \Db(S)$, there is a triangle:
\[
\syz^{i+1}(C) \longrightarrow P_i \longrightarrow \syz^i(C) \longrightarrow,
\]
where $P_i$ is projective. Applying $\mathcal{G}$ yields a triangle in $\Db(R)$:
\[
\mathcal{G}(\syz^{i+1}(C)) \longrightarrow \mathcal{G}(P_i) \longrightarrow \mathcal{G}(\syz^i(C)) \longrightarrow.
\]
Since $P_i$ is projective, $\mathcal{G}(P_i) \in \add T \subseteq \K_{[-k_T,0]}(\PR) \subseteq \K_{(-\infty,0]}(\PR)$. By Lemma \ref{lem:syzygy-comparison} with $t = 0$, we obtain for all $j \geq 0$:
\[
\syz^j(\mathcal{G}(\syz^{i+1}(C))) \simeq \syz^{j+1}(\mathcal{G}(\syz^i(C))).
\]

(2) The proof is dual to (1). By Lemma \ref{lem:syzygy-triangle}, for $C\in \Db(R)$, there is a triangle:
\[
\syz^{i+1}(C) \longrightarrow P_i \longrightarrow \syz^i(C) \longrightarrow,
\]
where $P_i$ is projective. Applying $\mathcal{F}$ yields a triangle in $\Db(S)$:
\[
\mathcal{F}(\syz^{i+1}(C)) \longrightarrow \mathcal{F}(P_i) \longrightarrow \mathcal{F}(\syz^i(C)) \longrightarrow.
\]
Since $P_i$ is projective, $\mathcal{F}(P_i) \in \mathcal{F}(\add R) \subseteq \K_{[0,k_T]}(\PS) \subseteq \K_{(-\infty,k_T]}(\PS)$. By Lemma \ref{lem:syzygy-comparison} with $t = k_T$, we obtain for all $j \geq k_T$:
\[
\syz^j(\mathcal{F}(\syz^{i+1}(C))) \simeq \syz^{j+1}(\mathcal{F}(\syz^i(C))).
\]
\end{proof}

\begin{corollary} \label{cor:iterative-syzygy}
$(1)$
\(
\syz^j(\mathcal{G}(\syz^i(C))) \simeq \syz^{j+i}(\mathcal{G}(C)) 
\) holds for $C \in \modd S$ and $i,j \geq 0$.

$(2)$ 
\(
\syz^j(\mathcal{F}(\syz^i(C))) \simeq \syz^{j+i}(\mathcal{F}(C))
\) holds for $C \in \modd R$, $j \geq k_T$ and $i\ge 0$.
\end{corollary}

\begin{proof}
Note that $C=\syz^0(C)$ and then apply the above lemma inductively. 
\end{proof}

\subsection{The key syzygy identity}

\begin{lemma} \label{lem:key-syzygy}
$(1)$ 
\(
\syz^{0}\bigl(\mathcal G(\syz^{j}(\mathcal F(X)))\bigr) \simeq \syz^{j}(X)
\) holds for any $X \in \mathcal D_{[a,b]}(R)$ and any integer $j \ge \max(b,0)$.

$(2)$ 
\(
\syz^{k_T}\bigl(\mathcal F(\syz^{j-k_T}(\mathcal G(X)))\bigr) \simeq \syz^{j}(X)
\) holds for any $X \in \mathcal D_{[a,b]}(S)$ and any integer $j \ge \max(b,k_T)$.
\end{lemma}

\begin{proof}
(1) Let $P$ be a projective resolution of $\mathcal F(X)$. Choose an integer $c'$ sufficiently negative so that the truncated complex
\[
Y := (\sigma_{[c',-1]}P)[1]\in \mathcal K_{[c'+1,\,0]}(\mathcal P_S)
\]
fits into the triangle
\[
\mathcal F(X) \longrightarrow \syz^0(\mathcal F(X)) \longrightarrow Y \longrightarrow \mathcal F(X)[1].
\]
Applying $\mathcal G$ gives
\[
X \longrightarrow \mathcal G(\syz^0(\mathcal F(X))) \longrightarrow \mathcal G(Y) \longrightarrow X[1].
\]
By Lemma~\ref{lem:image-projective}, we have $\mathcal G(Y)\in \mathcal K_{[c'+1-k_T,\,0]}(\mathcal P_R)$. Hence, for $j \ge \max(b,0)\ge 0$, the $j$-th syzygy $\syz^j(\mathcal G(Y))$ is projective (in fact, it is zero up to projective summands). Also, since $j\ge b$, $\syz^j(X)$ is an $R$-module by Lemma~\ref{lem:syzygy-basic}. Applying the syzygy functor $\syz^j$ to the above triangle yields an isomorphism
\[
\syz^j\bigl(\mathcal G(\syz^0(\mathcal F(X)))\bigr) \simeq \syz^j(X)
\]
in the stable category of $R$-modules.

On the other hand, Lemma~\ref{lem:syzygy-descent} with $C=\mathcal F(X)$, $i=0$, and the above $j$ gives
\[
\syz^{0}\bigl(\mathcal G(\syz^{j}(\mathcal F(X)))\bigr)
\simeq
\syz^{j}\bigl(\mathcal G(\syz^0(\mathcal F(X)))\bigr).
\]
Combining this with the previous isomorphism yields the desired identity.

(2) The proof is dual to (1).
\end{proof}

In particular, we obtain:

\begin{corollary}\label{cor:key-syzygy}
$(1)$ If $X \in \modd R$, then
\(
\syz^{0}(\mathcal{G}(\syz^{k_T}(\mathcal{F}(X)))) \simeq \syz^{k_T}(X).
\)

$(2)$ If $X \in \modd S$, then
\(
\syz^{k_T}(\mathcal{F}(\syz^{0}(\mathcal{G}(X)))) \simeq \syz^{k_T}(X).
\)
\end{corollary}

\begin{proof}
Take $a=b=0$ in Lemma \ref{lem:key-syzygy}, with $j=k_T$ for (1) and $j=k_T$ for (2).
\end{proof}

\subsection{The resolution transfer lemma}

The following lemma is the key construction for both the main theorem and the global invariant.

\begin{lemma} \label{lem:transfer}
$(1)$ Let \(N\in\modd R\) and set $
M_N:=\syz^{k_T}(\mathcal{F}(N))\in\modd S.
$
If \(M_N\) admits a \(k\)-\(\ddell\) resolution of length \(n\), then \(\syz^{k_T}(N)\) admits a \(k\)-\(\ddell\) resolution of length \(n\).

$(2)$ Let \(M\in\modd S\) and set $
N_M:=\syz^{0}(\mathcal{G}(M))\in\modd R.
$
If \(N_M\) admits a \(k\)-\(\ddell\) resolution of length \(n\), then \(\syz^{k_T}(M)\) admits a \(k\)-\(\ddell\) resolution of length \(n\).
\end{lemma}

\begin{proof} (1)
It is clear that $M_N\in\modd S$ since $\mathcal{F}(N)\in\D_{[0,k_T]}(S)$ by Lemma \ref{lem:image-bounded}.
Suppose $M_N$ has a classical $k$-$\ddell$ resolution of length \(n\) in $\modd S$:
\[
0 \longrightarrow D_n \longrightarrow D_{n-1} \longrightarrow \cdots \longrightarrow D_0 \longrightarrow M_N \longrightarrow 0. \tag{$\dagger$}
\]
By Definition \ref{def:classical-ddell}, for each \(i=0,\ldots,n\), there exists \(L_i\in\modd S\) such that
\[
\Omega^{n-i}(D_i) \oemb \Omega^{n+k}(L_i). \tag{$\ddagger$}
\]
That is, there are $Y_i$'s such that $\Omega^{n-i}(D_i)\oplus Y_i\simeq \Omega^{n+k}(L_i)$ in $\modd S$. The exact sequence consists of short exact sequences in $\modd S$ 
\[
    0\to C_{i+1}\to D_i\to C_i\to 0, 
\]
where $C_n=D_n$ and $C_0=M_N$.

Applying $\mathcal{G}$ to these exact sequences gives triangles in $\mathcal{G}(\modd S)\subseteq \D_{[-k_T,0]}(R)$:
 \[
    \mathcal{G}(C_{i+1})\to \mathcal{G}(D_i)\to \mathcal{G}(C_i)\to . 
\]
 
Now applying $\syz^{0}$, we obtain triangles in $\modd R$, and hence exact sequences in $\modd R$:
 \[
    0\to \syz^{0}(\mathcal{G}(C_{i+1}))\to \syz^{0}(\mathcal{G}(D_i))\to \syz^{0}(\mathcal{G}(C_i))\to 0,
\]
which combine to give an exact sequence in $\modd R$:
\[
0 \longrightarrow \syz^{0}(\mathcal{G}(D_n)) \longrightarrow \cdots \longrightarrow \syz^{0}(\mathcal{G}(D_0)) \longrightarrow \syz^{0}(\mathcal{G}(M_N)) \longrightarrow 0. \tag{$\star$}
\]

 Moreover, from \((\ddagger)\), applying $\mathcal{G}$ gives an embedding
\[
\mathcal{G}(\Omega^{n-i}(D_i)) \oemb \mathcal{G}(\Omega^{n+k}(L_i)).
\]
By Corollary \ref{cor:iterative-syzygy} (1), we obtain
\[
\syz^{0}(\mathcal{G}(\Omega^{n-i}(D_i))) \simeq \Omega^{n-i}\bigl(\mathcal{G}(D_i)\bigr) \simeq \Omega^{n-i}\bigl(\syz^{0}(\mathcal{G}(D_i))\bigr),
\]
and similarly,
\[
\syz^{0}(\mathcal{G}(\Omega^{n+k}(L_i))) \simeq \Omega^{n+k}\bigl(\mathcal{G}(L_i)\bigr) \simeq \Omega^{n+k}\bigl(\syz^{0}(\mathcal{G}(L_i))\bigr).
\]
Combining these yields
\[
\Omega^{n-i}\bigl(\syz^{0}(\mathcal{G}(D_i))\bigr) \oemb \Omega^{n+k}\bigl(\syz^{0}(\mathcal{G}(L_i))\bigr),
\]
which is precisely the required condition for \((\star)\) to be a $k$-$\ddell$ resolution in $\modd R$ (note that all $\syz^{0}(\mathcal{G}(L_i))\in \modd R$).

Since
\[
\syz^{0}(\mathcal{G}(M_N))
= \syz^{0}(\mathcal{G}(\syz^{k_T}(\mathcal{F}(N))))
\simeq \syz^{k_T}(N)
\] 
by Corollary \ref{cor:key-syzygy} (1), we see that \((\star)\) is a $k$-$\ddell$ resolution of length $n$ of $\syz^{k_T}(N)$.

(2) The proof is dual to (1).
\end{proof}

\subsection{Proof of the main theorem}

\begin{theorem} \label{thm:main}
Assume Setup \ref{setup:derived}. 

$(1)$ If every $S$-module admits a $(k+k_T)$-$\ddell$ resolution, then every $R$-module admits a $k$-$\ddell$ resolution. In particular, $(k+k_T)$-$\ddell (S)<\infty$ implies $k$-$\ddell (R)<\infty$.

$(2)$ If every $R$-module admits a $(k+k_T)$-$\ddell$ resolution, then every $S$-module admits a $k$-$\ddell$ resolution. In particular, $(k+k_T)$-$\ddell (R)<\infty$ implies $k$-$\ddell (S)<\infty$.
\end{theorem}

\begin{proof}
The theorem now follows from Lemmas \ref{lem:syz-k-ddell} and \ref{lem:transfer}. Indeed, if every $S$-module admits a $(k+k_T)$-$\ddell$ resolution, then $\syz^{k_T}(M)$ admits a $(k+k_T)$-$\ddell$ resolution for every $R$-module $M$ by Lemma \ref{lem:transfer} (1). It follows that $M$ admits a $k$-$\ddell$ resolution by Lemma \ref{lem:syz-k-ddell} (2). The second statement is proved similarly.
\end{proof}

Let us say that \emph{a module admits an $\infty$-$\ddell$ resolution of length $n$} if it admits a $k$-$\ddell$ resolution for every $k\ge 1$.

\begin{corollary}\label{cor:derived-invariance}
The property that every module admits a finite $\infty$-$\ddell$ resolution is invariant under derived equivalences. 
\end{corollary}

\begin{proof}
This follows by symmetry of the derived equivalence and Theorem \ref{thm:main}.
\end{proof}

An $R$-module $T$ is tilting if $T$ itself is quasi-isomorphic to a tilting complex. Two module categories $\modd R$ and $\modd S$ are tilting equivalent if there is a tilting $T\in\modd R$ such that $S\simeq \mathrm{End}_RT$. Thus, tilting equivalence is a special case of derived equivalence. So we have the following corollary directly.

\begin{corollary} \label{cor:tilting}
Let $T$ be a tilting $R$-module of projective dimension $n$ and $S = \End_R(T)$. Then 

$(1)$ $(k+n)\text{-}\ddell(R) < \infty$ implies  $k\text{-}\ddell(S) < \infty$ and $(k+n)\text{-}\ddell(S) < \infty$ implies  $k\text{-}\ddell(R) < \infty$.

$(2)$ $\infty\text{-}\ddell(R) < \infty$ if and only if  $\infty\text{-}\ddell(S) < \infty$.

\end{corollary}


\subsection{Global derived delooping level}\label{sec:gddell} 

We now introduce the global counterpart of Guo--Igusa's derived
delooping level. Recall that Barrios--Lanzilotta--Mata \cite{BLM}
defined the global delooping level associated with Gélinas's
invariant by
\[
\Dell(R):=\sup\{\dell(M)\mid M\in\modd R\}.
\]
Replacing $\dell$ with the finer invariant $\ddell$ leads to the
following refinement.

\begin{definition}\label{def:gddell}
For an Artin algebra $R$ and an integer $k\ge 1$, the \emph{global $k$-derived delooping level} of $R$ is defined as
\[
k\text{-}\gddell(R) := \sup\{ k\text{-}\ddell(M) \mid M\in \modd R\}.
\]
\end{definition}

As usual, we set \(\gddell(R) := 1\text{-}\gddell(R)\).

\begin{lemma}\label{lem:global-ddell}
Let \(R\) be an Artin algebra, let \(k\geq 1\), and let
\(\ell=\operatorname{LL}(R)\) be the Loewy length of \(R\). If
\[
m=k\text{-}\ddell(R)<\infty,
\]
then
\[
k\text{-}\gddell(R)
\leq \ell(m+1)-1.
\]
Consequently,
\[
k\text{-}\ddell(R)<\infty
\quad\Longleftrightarrow\quad
k\text{-}\gddell(R)<\infty.
\]
\end{lemma}

\begin{proof}
Let \(J=\operatorname{rad}R\). Since
\(\ell=\operatorname{LL}(R)\), we have \(J^\ell=0\).
For every \(M\in\modd R\), consider its Loewy
filtration
\[
0=MJ^\ell\subseteq MJ^{\ell-1}\subseteq\cdots
\subseteq MJ\subseteq M.
\]
Each quotient $
MJ^i/MJ^{i+1}$ 
is semisimple. By the definition of
\(k\text{-}\ddell(R)\), every simple right
\(R\)-module \(S\) satisfies $k\text{-}\ddell(S)\leq m.$ 
Moreover, the \(k\)-derived delooping level of a finite direct
sum is bounded above by the maximum of the \(k\)-derived
delooping levels of its summands; this follows by taking finite
direct sums of the defining resolutions. Hence
\[
k\text{-}\ddell
   \bigl(MJ^i/MJ^{i+1}\bigr)\leq m
\]
for every \(0\leq i<\ell\). We claim that, for \(1\leq t\leq\ell\),
\[
k\text{-}\ddell(M/MJ^t)
\leq t(m+1)-1.
\]
For \(t=1\), the module \(M/MJ\) is semisimple, and therefore
$k\text{-}\ddell(M/MJ)\leq m.$ 
Suppose that the assertion holds for \(t-1\). From the short
exact sequence
\[
0\longrightarrow MJ^{t-1}/MJ^t
\longrightarrow M/MJ^t
\longrightarrow M/MJ^{t-1}
\longrightarrow0
\]
and the extension inequality \cite[Lemma~3.1]{GuoIgusa}, whose
proof applies verbatim to Artin algebras,
\[
k\text{-}\ddell(Y)
\leq
k\text{-}\ddell(X)
+
k\text{-}\ddell(Z)+1
\]
for a short exact sequence
\(0\to X\to Y\to Z\to0\), we obtain
\begin{align*}
k\text{-}\ddell(M/MJ^t)
&\leq
k\text{-}\ddell(MJ^{t-1}/MJ^t)
+
k\text{-}\ddell(M/MJ^{t-1})+1\\
&\leq
m+\bigl((t-1)(m+1)-1\bigr)+1\\
&=t(m+1)-1.
\end{align*}
This proves the claim by induction. Taking \(t=\ell\) and using \(MJ^\ell=0\), we obtain $k\text{-}\ddell(M)
\leq\ell(m+1)-1$ 
for every \(M\in\modd R\). Therefore,
\[
k\text{-}\gddell(R)
\leq\ell(m+1)-1<\infty.
\]

Conversely, since every simple right \(R\)-module belongs to
\(\modd R\), one has
\[
k\text{-}\ddell(R)
\leq k\text{-}\gddell(R).
\]
The asserted equivalence now follows.
\end{proof}

The following result captures the gap between the global derived delooping levels of derived equivalent Artin algebras. 
\begin{proposition}\label{pro:gddell-finiteness}
Assume Setup~\ref{setup:derived}, and for  $k\geq 1$.

$(1)$ If $(k+k_T)\text{-}\gddell(S)<\infty,$  then $k\text{-}\gddell(R)<\infty.$

$(2)$ If $(k+k_T)\text{-}\gddell(R)<\infty,$ then $k\text{-}\gddell(S)<\infty.$\\
More precisely, we have  
\[ k\text{-}\gddell(R) \le (k+k_T)\text{-}\gddell(S) +k_T, and \]
\[k\text{-}\gddell(S) \le (k+k_T)\text{-}\gddell(R) +k_T. \]
\end{proposition}

\begin{proof}
By Lemma~\ref{lem:global-ddell}, for every Artin algebra $A$ and
every integer $k\geq 1$, one has
\[
k\text{-}\ddell(A)<\infty
\quad\Longleftrightarrow\quad
k\text{-}\gddell(A)<\infty.
\]
Thus (1) follows from Theorem~\ref{thm:main}(1), and (2) follows
from Theorem~\ref{thm:main}(2). 

Next, we prove \[k\text{-}\gddell(R) \le (k+k_T)\text{-}\gddell(S) +k_T.\tag{*}\]  The proof of $k\text{-}\gddell(S) \le (k+k_T)\text{-}\gddell(R) +k_T $ is symmetric  using Lemma~\ref{lem:transfer}(2). 

Let $M \in \modd R$ be arbitrary. Set 
\[
M_S := \syz^{k_T}(\mathcal{F}(M)) \in \modd S,
\]
which is indeed a module by Lemma~\ref{lem:image-bounded}. By definition of the global derived delooping level, we have
\[
(k+k_T)\text{-}\ddell(M_S) \le (k+k_T)\text{-}\gddell(S).
\]
Thus $M_S$ admits a $(k+k_T)$-$\ddell$ resolution of length $n \le (k+k_T)\text{-}\gddell(S)$.

Applying Lemma~\ref{lem:transfer}(1) with $N = M$ and $k$ replaced by $k+k_T$, we obtain that $\syz^{k_T}(M)$ admits a $k$-$\ddell$ resolution of length $n$. Hence
\[
k\text{-}\ddell(\syz^{k_T}(M)) \le n \le (k+k_T)\text{-}\gddell(S).
\]

By Lemma~\ref{lem:syz-k-ddell} (2) (taking $m = k_T$), we deduce that $M$ admits a $k$-$\ddell$ resolution of length at most
\[
k\text{-}\ddell(\syz^{k_T}(M)) + k_T \le (k+k_T)\text{-}\gddell(S) + k_T.
\]
Therefore
\[
k\text{-}\ddell(M) \le (k+k_T)\text{-}\gddell(S) + k_T.
\]
Taking the supremum over all $M$ yields (*). 
\end{proof}
\begin{corollary}\label{cor:gddell-finiteness}
Assume Setup~\ref{setup:derived}. Then
\[
\infty\text{-}\gddell(R)<\infty
\quad\Longleftrightarrow\quad
\infty\text{-}\gddell(S)<\infty.
\]
Whenever these equivalent conditions hold, one has
\[
\left|
\infty\text{-}\gddell(R)
-
\infty\text{-}\gddell(S)
\right|
\leq k_T.
\]
\end{corollary}

\begin{proof}
Taking the supremum over all $k\geq1$ in the first inequality of
Proposition~\ref{pro:gddell-finiteness}, and observing that
$\{k+k_T\mid k\geq1\}\subseteq\{j\mid j\geq1\}$, gives
\[
\infty\text{-}\gddell(R)
\leq
\infty\text{-}\gddell(S)+k_T
\]
and, symmetrically,
\[
\infty\text{-}\gddell(S)
\leq
\infty\text{-}\gddell(R)+k_T.
\]
The assertions follow.
\end{proof}
\section{Examples and applications}\label{sec:examples}

\subsection{Algebras with finite derived $\infty$-delooping level}

We first list some easy examples of algebras with finite   derived $\infty$-delooping level.

\begin{example}\label{exm:f-ddell}
We have $\infty$-$\ddell(\Lambda)<\infty$ if $\Lambda$ is in one of the following classes of algebras.

$(1)$ Algebras of finite global dimension.

$(2)$ Gorenstein algebras.
    
$(3)$ Algebras of finite representation type.
    
$(4)$ Syzygy-finite algebras, including monomial algebras and left or right serial algebras, etc.

$(5)$ Local algebras.

$(6)$ One-point extension algebras $\Lambda=A[M]$ with $\infty$-$\ddell(A) < \infty$.

$(7)$ Algebras with finite $\Dell$, including truncated path algebras \cite{BLM}.
\end{example}

\begin{proof}
(1) is a special case of (2).

(2) follows from the fact that every Gorenstein projective module is infinitely deloopable. Also (2) is a special case of (7) by \cite[Theorem 3.5]{BLM}.

(3) is clearly a special case of (4), and (4) is a special case of (7) by \cite[Proposition 3.9]{BLM}.

(5) is \cite[Theorem 3.5]{GuoIgusa}.
 
 (6) follows from \cite[the sentence after   Corollary 3.6]{GuoIgusa}. 

(7) follows from \cite[Proposition 3.7]{BLM}. Indeed, if $\Dell(\Lambda)\le m< \infty$, then $\Omega^m(M)$ is infinitely deloopable for any $M\in \modd \Lambda$. Then every $M$ admits a derived $\infty$-delooping resolution of length $m$. It follows that $\infty$-$\gddell(\Lambda)\le m$.
\end{proof}
\begin{remark}
  By Theorem~\ref{thm:main}, every algebra derived equivalent to one of the algebras in Example~\ref{exm:f-ddell} has finite derived $\infty$-delooping level. Lemma~\ref{lem:global-ddell}, together with Corollary~\ref{cor:gddell-finiteness}, then implies that it also has finite global derived $\infty$-delooping level.
\end{remark}
\subsection{Examples of derived equivalences} We first revisit Chen's counterexample to illustrate how our main
results may be applied even when the classical delooping invariants behave very differently under a derived equivalence.

\begin{example}
Recall from Example~\ref{ex:kershaw-rickard} that Chen constructed
derived equivalent algebras $B$ and $C$ satisfying
\[
\dell(B)=\subddell(B)=\infty,\qquad \dell(C)=\subddell(C)=0.
\]
Since Chen works with left modules whereas we work with right modules,
we pass to the opposite algebras. The algebra $B^{\op}$ is a one-point
extension of a local algebra and therefore has finite derived
$\infty$-delooping level. Moreover, $B^{\op}$ and $C^{\op}$ are derived
equivalent. It follows from Corollary~\ref{cor:derived-invariance} that
$C^{\op}$ also has finite derived $\infty$-delooping level.
\end{example}

\begin{example}
Let $\Bbbk$ be a field and, composing paths from left to right, put
\[
R=\Bbbk\left(1\xrightarrow{\alpha}2\xrightarrow{\beta}3\right) \qquad\text{and} \qquad
S=\Bbbk\left(1\xrightarrow{a}2\xrightarrow{b}3\right)
\big/\langle ab\rangle.
\]
The algebras $R$ and $S$ are derived equivalent via the classical tilting $R$-module $T=P_1\oplus S_1\oplus P_3$, with $S\cong\End_R(T)$.
A direct computation yields the following values for every $k\ge 1$:
\[
\begin{array}{c|cc}
 & R & S \\ \hline
k\text{-}\ddell & 1 & 2\\
k\text{-}\gddell & 1 & 2.
\end{array}
\]
Both invariants differ by exactly $1$ between the two derived equivalent algebras.
\end{example}

\begin{example}
Let $\Bbbk$ be a field and, composing paths from left to right, let
\[
R=\Bbbk Q/J^2,
\qquad
Q:
\begin{tikzcd}[column sep=3.6em]
1 \arrow[r,bend left=18,"\alpha"]
&2 \arrow[l,bend left=18,"\beta"]\arrow[r,"\gamma"]
&3,
\end{tikzcd}
\] 
and 
\[
S\cong
\Bbbk
\left(
\begin{tikzcd}[column sep=3.6em]
1\arrow[r,"a"]&2\arrow[r,"b"]&3\arrow[ll,bend left=35,"c"]
\end{tikzcd}
\right)
\big/\langle ca,\,abc\rangle.
\]
The algebras $R$ and $S$ are derived equivalent  via the classical tilting $R$-module $T=P_1\oplus P_2\oplus (P_2/L_3)$, with $S\cong\End_R(T)$.
A direct computation yields the following values for every $k\ge 1$:
\[
\begin{array}{c|cc}
 & R & S \\ \hline
k\text{-}\ddell & 0 & 1\\
k\text{-}\gddell & 1 & 2.
\end{array}
\]
Still both invariants differ by exactly $1$ between the two derived equivalent algebras.
\end{example}

\raggedbottom

\end{document}